\documentclass[11pt]{amsart}

\usepackage{amscd,amssymb,amsmath,latexsym,enumerate}
\usepackage[utf8]{inputenc}
\usepackage{amscd,amssymb,amsmath,amsthm,bbm}
\usepackage[mathscr]{euscript}
\usepackage{mathrsfs}

\usepackage[breaklinks=true,colorlinks=true,
linkcolor=black,urlcolor=black,citecolor=black,
bookmarks=true,bookmarksopenlevel=2]{hyperref}

\DeclareMathAlphabet{\mathpzc}{OT1}{pzc}{m}{it}
\numberwithin{equation}{section}
\newtheorem{definition}{Definition}[section]
\newtheorem{theorem}[definition]{Theorem}
\newtheorem{proposition}[definition]{Proposition}

\newtheorem{corollary}[definition]{Corollary}
\newtheorem{remark}[definition]{Remark}

\theoremstyle{remark}
\newtheorem{example}[definition]{Example}

       \def\vgf{\varphi}    
            \def\gl{\lambda}
\def\gm{\mu}                 
            
\def\gs{\sigma}

     \def\Gd{\Delta}

\def\Gw{\Omega}              
\newcommand{\diver}{\mathrm{div}\,}
\newcommand{\loc}{{\mathrm{loc}}}
\newcommand{\dx}{\,\mathrm{d}x}

\newcommand{\dm}{\,\mathrm{d}m}

\newcommand{\dmu}{\,\mathrm{d}\mu}

\newcommand{\core}{C_0^{\infty}(\Omega)}
\newcommand{\R}{{\mathbb R}}

\renewenvironment{proof}{{\bfseries Proof.}}{\hfill$\Box$}

\newcommand{\CM}{{\mathbb C}}
\newcommand{\NM}{{\mathbb N}}

\newcommand{\RM}{{\mathbb R}}

\newcommand{\Cc}{{\mathcal C}}
\newcommand{\Dd}{{\mathcal D}}
\newcommand{\Ee}{{\mathcal E}}
\newcommand{\Hh}{{\mathcal H}}

\newcommand{\ab}{{\mathbf a}}

\allowdisplaybreaks

\begin{document}
\title{Shnol-type theorem for the Agmon ground state}
\author{Siegfried Beckus, Yehuda Pinchover}

\address{Department of Mathematics\\
Technion - Israel Institute of Technology\\
Haifa, Israel}
\email{beckus.siegf@technion.ac.il}

\address{Department of Mathematics\\
Technion - Israel Institute of Technology\\
Haifa, Israel}
\email{pincho@technion.ac.il}

\begin{abstract}
Let $H$ be a Schr\"odinger operator defined on a noncompact Riemannian manifold $\Omega$, and let $W\in L^\infty(\Omega;\mathbb{R})$. Suppose that the operator $H+W$ is critical in $\Omega$, and let $\varphi$ be the corresponding Agmon ground state. We prove that if $u$ is a generalized eigenfunction of $H$ satisfying $|u|\leq \varphi$ in $\Omega$, then the corresponding eigenvalue is
in the spectrum of $H$. The conclusion also holds true if for some $K\Subset \Omega$ the operator $H$ admits a positive solution in $\tilde{\Omega}=\Omega\setminus K$, and $|u|\leq \psi$ in $\tilde{\Omega}$, where $\psi$ is a positive solution of minimal growth in a neighborhood of infinity in $\Omega$.

Under natural assumptions,  this result holds true also in the context of infinite graphs, and Dirichlet forms.\\[2mm]

\noindent  2000  \! {\em Mathematics  Subject  Classification.}
Primary  \! 35P05; Secondary  35B09, 35J10, 35R02, 39A12, 81Q10, 81Q35.\\[2mm]

\end{abstract}


\maketitle


\section{Introduction}
\label{Chap-Int}

In 1957, Shnol \cite{Shn57} proved that if a generalized eigenfunction of a Schr\"odinger operator $H$ on $\R^d$ with potential bounded from below has at most a polynomial growth, then the corresponding energy is in the spectrum of $H$. This celebrated result was independently rediscovered by Simon \cite{Sim81} for a more general class of potentials. Additionally, \cite{Sim81} proved that a dense subset of the spectrum of $H$ admits a polynomial bounded solution. More precisely, it is shown that there is a polynomially bounded solution of the equation $Hu=\lambda u$ in $\R^d$ for $H$-spectrally a.e. \!energy $\lambda$, see also \cite{CyconFroeseKirschSimon87,Shu92}. A remarkable generalization of Shnol's theorem in the Dirichlet form setting was proven in \cite{BoLeSt09}  (see also \cite{HaesKe11,FrLeWi14} and references therein for related results). Also, a converse of Shnol's theorem was proven in the Dirichlet setting \cite{BoSt03}. The proofs rely on a local estimate of the gradient of the generalized eigenfunction, which is needed to control the mixed terms. For this purpose, a Caccioppoli-type inequality is a crucial tool (see \cite{HeinonenKilpelainenMartio93,BiMo95}, and references therein for the unperturbed operator, and \cite{BoLeSt09} for the Dirichlet form setting).

The present work deals with a Shnol-type theorem on a noncompact manifold by replacing the polynomial bound with an object intrinsically defined through the Schr\"odinger operator. More precisely, let $H$ be a {\em critical} Schr\"odinger operator defined on a noncompact manifold, and denote by $\vgf$ its (Agmon) ground state. It is natural to believe that a Shnol-type theorem is still valid if the generalized eigenfunction is pointwise bounded by $\vgf$. This statement was conjectured by Devyver, Fraas and Pinchover in \cite[Conjecture~9.9]{DeFrPi14}. In this work, we provide a positive answer to this conjecture. Let us first provide the class of operators we are mainly interested in, before stating the main theorem. The reader is referred to \cite{Agm83,Pi07,Pinsky95} and references therein for a more detailed discussion on criticality theory.

Let $\Omega$ be a domain in $\R^d$ (or a noncompact $d$-dimensional connected Riemannian manifold). Let $m>0$ be a positive measurable function in $\Gw$, and denote $\dm: =m(x)\dx$, where $\dx$ is the volume form of $\Omega$ (which is just the Lebesgue measure in the case of Schr\"odinger operators on domains of $\R^n$). We assume that in any coordinate system $(U;x_{1},\ldots,x_{d})$ the operator $H$ has the form
\begin{equation}\label{H}
H \;
	:= \; -\diver\! (A\nabla) + V,
\end{equation}
where the minus divergence is the formal adjoint of the gradient with respect to the measure $m$, and $A$ is a measurable symmetric matrix valued function $A:U\to\RM^{d^2}$ satisfying that for every $K\Subset U$ there is a constant $\lambda_K\geq 1$ such that $\lambda_K^{-1} I \leq A(x) \leq \lambda_K I$, where $I$ denotes the $d$-dimensional identity matrix,  $A\leq B$ means that $B-A$ is a nonnegative definite matrix, and we use the notation $K\Subset U$ if $K$ is relatively compact in $U$. For $p> d/2$, let $V\in L^p_{\loc}(\Omega,\RM)$. Define the symmetric form
$$
\ab(u,v)=\ab_H(u,v) \;
	:= \; \int_\Omega
			\big(\langle A\nabla u, \nabla v\rangle + V u\overline{v}\big)
		\dm(x)
$$
on $ C_0^\infty(\Omega)$ associated with the Schr\"odinger operator $H$, where $C_0^\infty(\Omega)$ denotes the set of compactly supported smooth functions on $\Gw$.

Consider the induced quadratic form which we also denote by $\ab=\ab_H$. We say that $H$ is nonnegative in $\Gw$ (and write $H\geq 0$ in $\Gw$) if $\ab\geq 0$ on $\core$, and $H$ is {\em semibounded} if there is a constant $c\in\RM$ such that $\ab(v)\geq c \|v\|^2$ for all $v\in\core$, where $\|\cdot\|$ is the norm of $L^2(\Gw,\dm)$.

The operator $H$ is called {\em supercritical in $\Omega$} if $\ab$ is not nonnegative on $\core$. Suppose that $H\geq 0$ in $\Gw$. Then $H$ is said to be {\em critical in $\Gw$} if for any nonzero nonnegative $W\in L^p_{\loc}(\Omega,\RM)$, with $p>d/2$,  the operator $H-W$ is supercritical in $\Gw$, otherwise, $H$ is {\em subcritical} in $\Gw$.

We note that if $H$ is critical in $\Gw$, then (up to a multiplicative constant) the equation $Hu=0$ in $\Omega$ admits a unique positive (super)solution. Such a solution is called {\em Agmon ground state} (or in short, a ground state), see Theorem~\ref{t:char_criticality}.

It is well-known that if $\ab$ is nonnegative on $\core$, then the Dirichlet problem admits a unique solution in every bounded subdomain $\Gw'\Subset \Gw$ \cite{Agm83}. Hence, for $u\in\core$ we have $\ab(u)=0$ if and only if $u=0$. In particular, $\ab$ defines a scalar product on $ C_0^\infty(\Omega)$. Let $\mathcal{H}_\ab(\Omega)$ be the closure of $ C_0^\infty(\Omega)$ with respect to the norm induced by $\ab$. In the sequel it will be evident that the operator $H$ is critical in $\Gw$ if and only if $\mathcal{H}_\ab(\Omega)$ is not a functional space.

A nonzero function $u\in W^{1,2}_\loc(\Omega)$ is called a {\em generalized eigenfunction of $H$ with eigenvalue $\lambda\in\RM$} if $\ab(u,v)=\lambda  \int_\Omega u  \overline{v}\dm(x)$ for every $v\in C_0^\infty(\Omega)$. The main result of this work is the following.

\begin{theorem}
\label{Theo-Shnol}
Let $H$ be an operator of the form \eqref{H}, and assume that $H+W$ is critical in $\Gw$ for some real-valued $W\in L^\infty(\Gw)$. Denote by $\varphi$ the corresponding ground state of $H+W$. If $u$ is a generalized eigenfunction of $H$ with eigenvalue  $\lambda\in\RM$ satisfying $|u|\leq \varphi$ in $\Gw$, then $\lambda$ belongs to the $L^2(\Gw,\dm)$-spectrum of $H$.
\end{theorem}

\begin{remark}\label{rem1}{\em
In \cite{Pi07a} we prove a sufficient condition, in terms of the behavior of a
ground state of a critical Schr\"odinger operator $H_1$, such that a positive (sub)solution of a
Schr\"odinger operator $H$ is a ground state. In a certain sense, Theorem~\ref{Theo-Shnol} is a generalization of the main result of \cite{Pi07a}.
}
\end{remark}

If $H$ is subcritical, or more generally, the equation $Hw=0$ admits a positive solution in $\Gw\setminus K$, where $K\Subset \Gw$, then there exists a real-valued $W\in C_0(\Omega)$ such that $H+W$ is critical \cite{Pi07,Pinsky95}, where $C_0(\Omega)$ is the space of all continuous compactly supported functions on $\Omega$. Hence, Theorem~\ref{Theo-Shnol} applies in this case. In particular, we have the following result.

\begin{corollary}
\label{Cor-Shnol-sub}
Let $H$ be an operator of the form \eqref{H}, and assume that $H$ is subcritical in $\Gw$.  Let $x_0\in \Gw$ be a fixed reference point, and denote by $G(x):=G_H^\Gw(x,x_0)$ the positive minimal Green function of $H$ in $\Gw$ with a pole at $x_0$. If $u$ is a generalized eigenfunction of $H$ with eigenvalue  $\lambda\in\RM$, and satisfies $|u|\leq G$ in $\Gw\setminus B(x_0,r)$ for some $r>0$, then $\lambda$ belongs to the $L^2$-spectrum of $H$.
\end{corollary}
\begin{remark}\label{rem2}{\em
If $H$ is supercritical and semibounded, then there exists a constant $c>0$ such that $\tilde{H}:=H+cI\geq 0$ on $\core$. Then either Theorem~\ref{Theo-Shnol} or Corollary~\ref{Cor-Shnol-sub} clearly applies to $\tilde{H}$ and hence to generalized eigenfunctions of $H$.
 }
\end{remark}


The proof of Theorem~\ref{Theo-Shnol} hinges on a Weyl-type argument combined with a ground state transform (see Section~\ref{Chap-NullSeqGrdSt}) which eliminates the potential $V$.
Let us provide some examples where our main result (Theorem~\ref{Theo-Shnol}) applies.
\begin{example}\label{ex1}
Let $\Omega\subseteq\RM^d$ be a domain or a noncompact Riemannian manifold as above. Assume that the operator $H:=-\diver\! (A\nabla)$ on $L^2(\Omega,\dm)$ is critical in $\Gw$. The corresponding form acts on $\core$ as
$$
\ab(u)
	:=  \int_\Omega \langle A \nabla u, \nabla u\rangle \dm.
$$
Obviously, the ground state of $H$ in $\Gw$ is the constant function $1$. It follows from Theorem~\ref{Theo-Shnol} that if $u_\gl$ is a {\em bounded} generalized eigenfunction of $H$ with an eigenvalue $\gl$, then $\gl\in\gs(H)$. We note that particular cases of the present example are clearly given by parabolic Riemannian manifolds or more generally by critical weighted Laplace-Beltrami operators.
\end{example}
The aim of the following example is to provide a simple proof of the well known fact that the spectrum of the Laplacian on $\R^d$ is equal to $[0,\infty)$.
\begin{example}\label{ex11}
Let $H$ be the Laplacian on $\R$. Then $H$ is critical in $\R$ with a ground state equal $1$. Hence, $\sigma(H)\subset [0,\infty)$. Moreover, for any $k\in\R$, the exponential function $\mathrm{e}^{\mathrm{i}kx}$ is a bounded generalized eigenfunction with eigenvalue $\gl=k^2$. Theorem~\ref{Theo-Shnol} implies that $\sigma(H) = [0,\infty)$.
In particular, for any $\gl\geq 0$ there exists a Weyl sequence $\{\phi_{i,\gl}\}$ for the operator $H-\gl$.

Consider now the Laplacian $-\Gd$ on $\R^d$, and denote by $x=(x_1,\ldots,x_d)$ a point in $\R^d$.   It can be easily verified that the sequence $\{\prod_{j=d}^n\phi_{i,\gl}(x_j)\}$ is a Weyl sequence for $-\Delta-d\gl$ in $\R^d$, and hence, the spectrum of the Laplacian on $\R^d$ equals to $[0,\infty)$.
\end{example}
\begin{example}
Let $H$ be a symmetric subcritical operator on $\Gw$ of the form \eqref{H}. Let $W\geq 0$ be the optimal Hardy-weight that is given in \cite[Theorem~2.2]{DeFrPi14}. Assume further that $W>0$ in $\Gw$. Then the operator $W^{-1}H-1$ is critical in $\Gw$ with a ground state $u_0$, and  $u_0\not \in L^2(\Gw,W\dx)$. Moreover, for any $\gl\geq 1$ the operator $W^{-1}H$ admits generalized eigenfunction $u_\gl$ satisfying $|u_\gl| \leq u_0$ in $\Gw$, see \cite[Lemma~7.1 and the proof of Theorem~7.2]{DeFrPi14}.

Theorem~\ref{Theo-Shnol} implies that $\gs(W^{-1}H)=[1,\infty)$. This gives a simple alternative proof for the last statement of  \cite[Theorem~2.2]{DeFrPi14}.
\end{example}

The outline of the present paper is as follows. In Section~\ref{Chap-Set}, we provide a short summary on the theory of forms and in particular, we present a Weyl sequence assertion for forms. Section~\ref{Chap-NullSeqGrdSt} is devoted to a discussion of the basic notions and results we are using. The Caccioppoli type estimate is proven in Section~\ref{Chap-Caccio}. The proof of the main theorem is then presented in Section~\ref{Chap-Shnol}. Finally, in Section~\ref{Chap-Dir} we shortly discuss Dirichlet forms and the possible extension of our result to this setting, where we mainly focus on discrete Schr\"odinger operators on infinite graphs.


\section{Weyl-type argument}
\label{Chap-Set}
In the present section we consider a more general situation than considered hitherto. We provide a short summary of the concepts we are using. Let $\Omega$ be a locally compact, separable and connected metric space and $m$ be a  positive Radon measure with support $\Omega$. Let $f,g:\Gw \to (0,\infty)$. We write $f \asymp g$ in $\Gw$ if there exists a positive constant $C$ such that $C^{-1}g \leq f \leq Cg$ in $\Gw$.

\medskip

In order to prove our main result we use a Weyl sequence argument. In the general situation we are dealing with, it is convenient to work with the form associated with the operator. Recall that a {\em form} $q:\Dd(q)\times \Dd(q)\to\CM$ defined on a linear subspace of a (complex) Hilbert space $\Hh$ is linear in the first component and complex linear in the second component. Furthermore, $q$ is called {\em symmetric} if $q(v,w)=\overline{q(w,v)}$ for all $v,w\in\Dd(q)$. It is called {\em semibounded} if there is a constant $c\in\RM$ such that $q(v,v)\geq c \|v\|^2$ for all $v\in\Dd(q)$. If $q$ is a symmetric semibounded form with a constant $c$, then  $\|v\|^2_q:= q(v,v) + (1-c)\|v\|^2$ defines a norm on $\Dd(q)$ satisfying the parallelogram identity. We mainly deal with forms $q$ that are nonnegative ($q\geq 0$), and hence  $q$ is semibounded, where the constant $c$ is equal to zero. The form $q$ is called {\em closed} if $(\Dd(q),\|\cdot\|_q)$ is additionally complete. Every symmetric semibounded form defines a quadratic form by $q(u):= q(u,u)$ for $u\in\Dd(q)$. For every closed, symmetric, nonnegative form there is a unique self-adjoint operator $H$ associated with $q$. The spectrum of $H$ is denoted by $\sigma(H)$. We denote by $\Dd_{\loc}(q)$ the local form domain of $q$, see \cite{FrLeWi14}.

\medskip

The following proposition is a Weyl sequence assertion stated for forms. The result is proven in \cite[Lemma~1.4.4]{Stollmann01}, see also \cite[Proposition~2.1]{BoLeSt09}. We provide a proof for the convenience of the reader.

\begin{proposition}[\cite{Stollmann01}]
\label{Prop-CritWeyl}
Let $q:\Dd(q)\times \Dd(q)\to\CM$ be a closed, symmetric, semibounded form, and let $H$ be the associated self-adjoint operator. Then the following assertions are equivalent:
\begin{itemize}
\item[(i)] $\lambda\in\sigma(H)$
\item[(ii)] There exists a sequence $(w_n)$ in $\Dd(q)$ with $\lim\limits_{n\to\infty}\|w_n\|=1$ such that
\begin{equation}\label{Ws}
\lim_{n\to\infty} \;
	\sup_{v\in\Dd(q), \|v\|_q\leq 1}
		\Big|
			q(w_n,v)-\lambda  \langle w_n,v\rangle
		\Big|
			= 0.
\end{equation}
\end{itemize}
\end{proposition}

\begin{proof}
(i)$\Rightarrow$ (ii): Suppose that $\lambda$ admits a normalized eigenfunction $u$ of $H$, then set $w_n:=u$  and (ii) follows. Otherwise, let $(\tilde{w}_n)$ be a (classical) Weyl sequence of the operator $H$ for $\lambda$. In general, $(\tilde{w}_n)$ is not contained in the domain $\Dd(q)$. However, it is straightforward to find an approximating sequence $(w_n)$ satisfying (\ref{Ws}) by using that $\Dd(q)$ is a core of $\Dd(H)$.

(ii)$\Rightarrow$ (i): Let $(w_n)$ be a sequence in $\Dd(q)$ satisfying $\lim_{n\to\infty}\|w_n\|=1$ and \eqref{Ws}. Assume, to the contrary, that $\lambda\in\rho(H)$. The boundedness of the resolvent $(H-\lambda)^{-1}$ implies the finiteness of $C:=\sup_{n\in\NM} \|(H-\lambda)^{-1} w_n\|_q$. Since $q(u,v)=\langle Hu,v\rangle$ for $u,v\in D(q)$, we derive with the previous considerations
$$
\|w_n\|^2
	= \big|
			q(w_n,(H-\lambda)^{-1}w_n)-\lambda  \langle w_n,(H-\lambda)^{-1}w_n\rangle
		\big|
	\leq C \sup_{v\in\Dd(q), \|v\|_q\leq 1}
		\Big|
			q(w_n,v)-\lambda  \langle w_n,v\rangle
		\Big|
	\,.
$$
The term on the right hand side tends to zero if $n\to\infty$, contradicting $\lim\limits_{n\to\infty}\|w_n\|=1$.
\end{proof}


\section{Null-sequence and the ground state transform}
\label{Chap-NullSeqGrdSt}
Let $\Omega$ be a domain in $\R^d$ (or a noncompact $d$-dimensional connected Riemannian manifold). Let $\ab:=\ab_H$ be the form associated with the Schr\"odinger operator $H= -\diver (A\nabla) + V$ in $\Gw$.
A nonzero function $u\in W^{1,2}_\loc(\Omega)$ is called {\em $H$-harmonic}  in $\Gw$ if $\ab(u,v)=0$ for every $v\in\core$. A function $u\in W^{1,2}_\loc(\Omega)$ is called {\em $H$-superharmonic}  in $\Gw$ if $\ab(u,v)\geq 0$ for every nonnegative function $v\in\core$.
The cone of all positive $H$-harmonic functions in $\Omega$ is denoted by $\Cc_H(\Omega)$.

Next, we introduce the notion of a ground state which is a central object of our considerations.
\begin{definition}
\label{mg}{\em
Let $K\Subset \Gw$. A positive $H$-harmonic function $u$ in $\Gw\setminus K$ is called a {\em positive $H$-harmonic of minimal growth in a neighborhood of infinity in $\Omega$} if for every $K\Subset K'\Subset \Omega$ with smooth boundary and each positive $H$-superharmonic $v\in C(\overline{\Omega\setminus K'})$, the inequality $u\leq v$ on the boundary $\partial K'$ implies $u\leq v$ in $\Omega\setminus K'$.  A positive solution $u\in\Cc_H(\Omega)$ that has minimal growth at infinity in $\Omega$ is called a {\em ground state} of $H$ in $\Omega$.}
\end{definition}

\begin{definition}
\label{Def-NullSeq}{\em
Let $H$ be an operator of the form \eqref{H}, and assume that the quadratic form $\ab:=\ab_H$ is nonnegative on $\core$. A sequence $(\varphi_n)\subset  C_0^\infty(\Omega)$ of nonnegative functions is called a {\em null-sequence} for $\ab$ if there exists a ball $B\Subset \Gw$ such that $\int_B\varphi_n^2 \dm=1$ and $\lim_{n\to\infty}\ab(\varphi_n)=0$. Furthermore, a positive function $\varphi\in\Dd_{\loc}$ is called a {\em null-state} of $\ab$ if there exists a null-sequence $(\varphi_n)$ such that $(\varphi_n)$ converges in $L^2_{\loc}(\Omega)$ to $\varphi$.}
\end{definition}

The following characterization of criticality is well known in
various contexts, we refer to the following papers \cite{Pi07,PiTi06,Pinsky95,KePiPo16} and references therein.
\begin{theorem}[Criticality characterization]\label{t:char_criticality}
Let $H$ be an operator of the form \eqref{H}, and assume that the quadratic form $\ab_H$ is nonnegative on $\core$. Then, the following assertions are equivalent:
\begin{itemize}
  \item [(i)] $H$ is critical in $\Gw$.
  \item [(ii)]  $H$ admits a ground state in $\Gw$.
    \item [(ii)] $H$ does not admit a positive minimal Green function  in $\Gw$.
    \item [(iii)] $H$ admits a unique  (up to a multiplicative constant) positive $H$-superharmonic function  in $\Gw$.
\item [(iv)]  For any open ball $B\Subset \Gw$, there is a null-sequence $(\vgf_{n})$ such that $\int_B \vgf_{n}(x)^2\dm=1$ for all $n\geq 0$.
\item [(v)] There  exists a null-sequence $(\vgf_{n})$ satisfying $0\leq \vgf_{n}\leq v$ in $\Gw$, where $v$ is a positive $H$-harmonic function on $\Gw$, and $\vgf_{n}(x)\to v(x)$  locally uniformly in $\Gw$.
\end{itemize}
In particular, $\vgf$ is a null-state if and only if it is a ground state.
\end{theorem}
Note that $\dim\big(\Cc_H(\Omega)\big)=1$ holds whenever $H$ is critical. Thus, a ground state is unique (up to a multiplicative constant).

Another important ingredient is the ground state transform. For a positive continuous function $h\in W^{1,2}_\loc(\Omega)\cap C(\Omega)$, define
$$T_h(v):=\frac{v}{h}\,.$$ Note that $T^{-1}_h=T_{h^{-1}}$ holds.
The operator $$H_h:= T_h\circ H\circ T^{-1}_h$$ is called the {\em $h$-transform} of $H$. If $\varphi\in \Cc_H(\Omega)$, then $H_\varphi$ is called a {\em ground state transform} of $H$. Furthermore, $T_\varphi|_{\Dd_{\loc}\cap L^2(\Omega)}$ extends to an isometry between $L^2(\Omega,\dm)$ and $L^2(\Omega,\varphi^2 \dm)$, see for example \cite[Proposition~4.15]{DeFrPi14}. Hence, $\sigma(H)=\sigma(H_\varphi)$ follows. The reader is referred to \cite{DeFrPi14,KePiPo16} and references therein for further background on the ground state transform.

Consider the Schr\"odinger operator $H= -\diver (A\nabla) + V$ on $\Gw$ with associated nonnegative form
$$
\ab(u,v) \;
	:= \; \int_\Omega
			\left(\langle A\nabla u, \nabla v\rangle + V u\overline{v}\right)
		\dm
	\,.
$$
Suppose that $H$ is critical in $\Gw$ with ground state $\varphi$. Then the quadratic form $\ab_\varphi$ of the operator $H_\varphi$ is given by
$$
\ab_\varphi(u,v) \;
	:= \; \int_\Omega \langle A\nabla u, \nabla v\rangle\dmu (x),
$$
where the measure $\gm$ is given by $\varphi^2\dm$. For more details see for example \cite{DeFrPi14}. It is worth mentioning that $H_\varphi$ is critical in $\Gw$, and its ground state is given by the constant function $1$. This fact will be intensively used. Finally, let us point out that we obtain the following Cauchy-Schwarz inequality
\begin{align*} \tag{C.S.}
\left|
	\int_\Omega u v \langle A\nabla v, \nabla u\rangle \dmu
\right|
	&\leq \int_\Omega |u| |\langle A\nabla v, \nabla v\rangle|^{\frac{1}{2}}  |v|  |\langle A \nabla u, \nabla u\rangle|^{\frac{1}{2}} \dmu\\
	&\leq \left(\int_\Omega |u|^2 \langle A\nabla v, \nabla v\rangle \dmu\right)^{\frac{1}{2}}
		\left(\int_\Omega |v|^2 \langle A\nabla u, \nabla u\rangle \dmu\right)^{\frac{1}{2}}
\end{align*}
by using that $A$ is nonnegative definite.

\section{Caccioppoli-type estimate}
\label{Chap-Caccio}
Throughout this section we consider a nonnegative form $\ab$ given by
$$
\ab(u,v) \;
	:= \; \int_\Omega \langle A\nabla u, \nabla v\rangle\ d\mu (x)
	\,,
$$
with the associated self-adjoint operator $L:=-\diver\! (A\nabla)$ on $L^2(\Gw,\dmu)$. For a real-valued potential $W\in L^\infty(\Gw)$, consider the self-adjoint operator $L+W:=-\diver\! (A\nabla)+W$ on $L^2(\Gw,\dmu)$. The associated symmetric form $q$ of $L+W$ is defined by $q(u,v) = \ab(u,v) + \langle W u, v \rangle$, where $\langle \cdot, \cdot \rangle$ and $\|\cdot\|$ are the inner product and the norm on $L^2(\Gw,\dmu)$. Clearly, $q$ is semibounded with constant $c:=-\|W\|_\infty$ as $\ab$ is nonnegative. Thus, $\|\cdot\|_q$ is defined by $\|v\|_q^2 := q(v,v) + (1+\|W\|_\infty) \|v\|^2$ (see Section~\ref{Chap-Set}). We have the following Caccioppoli-type estimate.
\begin{proposition}
\label{Prop-Caccio}
Let $W\in L^\infty(\Gw)$ be real-valued, and let $u\in W^{1,2}_\loc(\Omega)$ be a generalized eigenfunction of the operator  $L+W$ with eigenvalue $\lambda\in\RM$ satisfying $|u|\leq 1$ in $\Gw$. Then
$$
\int_\Omega |v|^2  \langle A\nabla u,\nabla u\rangle \dmu \;
	\leq \; \Big(1 + \sqrt{1+|\lambda|+\|W\|_\infty}\Big)^2
$$
holds for every $v\in\Dd(q)$ satisfying $\|v\|_q\leq 1$.
\end{proposition}

\begin{proof}
We first show the desired estimate for $v\in\core$ with $\|v\|_q\leq 1$. Set
$$
z \;
	:= \; \sqrt{\int_\Omega |v|^2  \langle A\nabla u, \nabla u\rangle\dmu}
	\,.
$$
The constraint $\|v\|_q\leq 1$ implies $\|v\|\leq 1$ and $\ab(v,v)\leq 1$ since $\langle W \, v,v\rangle + \|W\|_\infty \|v\|^2\geq 0$. Thus, a short computation yields
\begin{align*}
z^2 \;
	&= \; \left|
			\int_\Omega \big\langle A\nabla \big(\overline{v}\, v \, u\big),\nabla u \big\rangle \dmu
			-
			\int_\Omega u \big\langle A\nabla \big(\overline{v}\, v\big),\nabla u \big\rangle \dmu
		\right|\\
	&\leq \;
		\bigg|
			\int_\Omega
				(\lambda  - W) \, |v|^2  \underbrace{|u|^2}_{\leq 1}
			\dmu
		\bigg|
		+
		\left|
			\int_\Omega
				u\, \overline{v}  \langle A\nabla v, \nabla u\rangle
			\dmu
		\right|
		+
		\left|
			\int_\Omega
				u\, v  \langle A\nabla \overline{v}, \nabla u\rangle
			\dmu
		\right|\\
	&\!\!\overset{\text{C.S.}}{\leq} \; (|\lambda| + \|W\|_\infty) \, \underbrace{\|v\|^2}_{\leq 1}
		+
		2\, \underbrace{
			\left(
				\int_\Omega
					|u|^2  \langle A\nabla v, \nabla v\rangle
				\dmu
			\right)^{\frac{1}{2}}
		}_{\leq \sqrt{\ab(v,v)}\leq 1}
				\left(
			\int_\Omega
				|v|^2  \langle A\nabla u, \nabla u\rangle
			\dmu
		\right)^{\frac{1}{2}}\\
	&\leq \; |\lambda| + \|W\|_\infty + 2 \, z
	\,.
\end{align*}
Since $v$ is compactly supported, the above integrals are all well-defined. An elementary computation implies $0 \leq z \leq 1 + \sqrt{1+|\lambda|+\|W\|_\infty}$ showing the desired estimate for $v\in C_0^\infty(\Omega)$ with $\|v\|_q\leq 1$.

\medskip

Now consider $v\in\Dd(q)$ satisfying $\|v\|_q\leq 1$. Since $ C_0^\infty(\Omega)$ is dense in $\Dd(q)$ with respect to $\|\cdot\|_q$, there exists a sequence $(v_n)\subseteq C_0^\infty(\Omega)$ such that $\lim_{n\to\infty}\|v-v_n\|_q=0$ and $\|v_n\|_q\leq 1$. Thus, $v-v_n$ tends to zero in the $L^2$-norm and so there is no loss of generality in assuming that $(v_n)$ converges $\dmu$-a.e. to $v$. Applying Fatou's Lemma, we get
$$
\int_\Omega |v|^2  \langle A\nabla u,\nabla u\rangle \dmu \;
	\leq \; \liminf_{n\to\infty} \int_\Omega |v_n|^2 \langle A\nabla u,\nabla u\rangle \dmu \;
	\leq \; \Big(1 + \sqrt{1+|\lambda|+\|W\|_\infty}\Big)^2.
$$
Thus, the proposition is proved.
\end{proof}

\section{Shnol-type theorem}
\label{Chap-Shnol}
Throughout this section we consider a nonnegative form $\ab$ given by
$$
\ab(u,v) \;
	:= \; \int_\Omega \langle A\nabla u, \nabla v\rangle\dmu (x)
	\,,
$$
with associated self-adjoint operator $L:=-\diver\! (A\nabla)$ on $L^2(\Gw,\dmu)$. Like in the previous section, $q$ denotes the form associated with $L+W$ for a real-valued potential $W\in L^\infty(\Gw)$.

\begin{proposition}
\label{Prop-Shnol}
Suppose that $L$ is critical in $\Gw$ with the ground state $1$. Let $W\in L^\infty(\Gw)$ be real-valued and $u$ be a generalized eigenfunction $L +W$ with eigenvalue $\lambda\in\RM$. If $|u|\leq 1$ in $\Gw$, then $\lambda\in\sigma(L+W)$.
\end{proposition}
\begin{proof}
Invoking Theorem~\ref{t:char_criticality}, there exists a null-sequence $(\varphi_n)$ of $L$ satisfying $0\leq \varphi_n\leq 1$ such that $(\varphi_n)$ converges, locally uniformly, to the ground state $1$. If $u\in L^2(\Gw,\dmu)$, then obviously, $\lambda\in\sigma(L+W)$.

So, we may suppose that $u\not\in L^2(\Gw,\dmu)$. Thus, $\lim_{n\to\infty}\|\varphi_n u\|=\infty$ follows by Fatou's Lemma since $\varphi_n u$ converges locally to $u$. Hence, there is an $n_0\in\NM$ such that $\|\varphi_n u\| \geq 1$ for $n\geq n_0$. Define $w_n:=\frac{\varphi_n u}{\|\varphi_n u\|}\,,\; n\geq n_0\,,$ so, $\|w_n\|=1$.

Denote by $q$ the symmetric semibounded form associated with $L+W$. Let $v\in\Dd(q)$ be so that $\|v\|_q\leq 1$. Since $u$ is a generalized eigenfunction of $L+W$, we have
$$
\lambda  \langle w_n,v\rangle
    = \frac{\lambda}{\|\varphi_n u\|}  \langle u,\varphi_n v\rangle
	=  \frac{1}{\|\varphi_n u\|} \, q(u,\varphi_n v) \;
	=  \frac{1}{\|\varphi_n u\|} \, \ab(u,\varphi_n v) + \langle Ww_n,v\rangle.
$$
Hence,  $\big|q(w_n,v) - \lambda  \langle w_n,v\rangle\big| = \frac{1}{\|\varphi_n u\|} \big|\ab(\varphi_n u,v) - \ab(u,\varphi_n v)\big|$. Therefore, the Leibniz rule implies
$$
\Big|
	q(w_n,v) - \lambda  \langle w_n,v\rangle
\Big|\;
	=\; \frac{1}{\|\varphi_n u\|}
		\left|
			\int_\Omega u 1  \langle A\nabla\varphi_n, \nabla v \rangle \dmu
			-
			\int_\Omega \bar{v} 1  \langle A\nabla u, \nabla \varphi_n \rangle \dmu
		\right|.
$$
Applying the triangle and the Cauchy-Schwarz inequalities, the previous terms are estimated by
\begin{align*}
&\frac{1}{\|\varphi_n u\|} \left[
	\left(
		\int_\Omega
			|u|^2  \langle A\nabla\varphi_n, \nabla \varphi_n \rangle
		\dmu
	\right)^{\frac{1}{2}}
	\ab(v,v)^{\frac{1}{2}}
	+
	\left(
		\int_\Omega
			|v|^2  \langle A\nabla u, \nabla u \rangle
		\dmu
	\right)^{\frac{1}{2}}
	\ab(\varphi_n,\varphi_n)^{\frac{1}{2}}
\right]\\
\leq \; & \frac{\ab(\varphi_n,\varphi_n)^{\frac{1}{2}}}{\|\varphi_n u\|}
		\left(
			1+
			\left(
				\int_\Omega
					|v|^2  \langle A\nabla u, \nabla u \rangle
				\dmu
			\right)^{\frac{1}{2}}
		\right),
\end{align*}
where $|u|\leq 1$ and $\|v\|_q\leq 1$ is used to obtain the second line.

Due to Proposition~\ref{Prop-Caccio}, the integral $\int_\Omega |v|^2  \langle A\nabla u, \nabla u \rangle \dmu$ is bounded by a constant (depending on $\|W\|_\infty$ and $\lambda$), and in particular, it is well-defined. Since $\|\varphi_n u\| \geq 1$ for $n\geq n_0$, we deduce
$$
\Big|
	q(w_n,v)-\lambda  \langle w_n,v\rangle
\Big|
	\!\leq \! \big(2\!+\!\sqrt{1+|\lambda|+\|W\|_\infty}\big)
			\ab(\varphi_n,\varphi_n)^{\frac{1}{2}},
		\qquad
$$
for all $n\geq n_0$ and $v\in\Dd(q)$ with $\|v\|_q\leq 1$. Since $(\varphi_n)$ is a null-sequence of $L$, the previous estimate implies that $\big|q(w_n,v)-\lambda  \langle w_n,v\rangle \big|$ tends to zero uniformly in $v\in\Dd(q)$ with $\|v\|_q\leq 1$. Additionally, $q$ is a closed, symmetric, semibounded form since $\ab$ is nonnegative and $W\in L^\infty(\Gw)$ is real-valued. Consequently, $\lambda\in\sigma(L+W)$ follows by Proposition~\ref{Prop-CritWeyl}.
\end{proof}

\medskip

Now, the main theorem is a direct consequence of the previous considerations.

\medskip

\begin{proof}[Proof of Theorem~\ref{Theo-Shnol}]
Since $H+W$ is critical, Theorem~\ref{t:char_criticality} implies that $H+W$ admits a ground state $\varphi$. Using the ground state transform $T_\varphi$, the operators $L-W$ and $H$ are unitary equivalent where $L:=T_\varphi\circ (H+W)\circ T_\varphi^{-1}$. Furthermore, $T_\varphi\varphi= 1$ is the ground state of the critical operator $L$. Set $u_\varphi=T_\varphi u$. Then $u_\varphi$ is a generalized eigenfunction of $L-W$ with eigenvalue $\lambda\in\RM$ and $|u_\varphi|\leq 1$. Applying Proposition~\ref{Prop-Shnol} to $L$ with generalized eigenfunction $u_\varphi$ of $L-W$ leads to $\lambda\in\sigma(L-W)=\sigma(H)$.
\end{proof}

\medskip

\begin{proof}[Proof of Corollary~\ref{Cor-Shnol-sub}]
According to \cite[Theorem~4.6]{Pi88} there is a nonnegative $W\in C_0(\Gw)$ such that $H-W$ is critical in $\Gw$, and let $\vgf$ be the corresponding ground state. Since $G \asymp \vgf$ \cite[Corollary~4.3]{Pi88} in $\Gw\setminus B(x_0,r)$, it follows that $|u|\leq C\vgf$ in $\Gw$ for some $C>0$. Hence, the corollary follows from Theorem~\ref{Theo-Shnol} applied to the generalized eigenfunction $u/C$ of $H$.
\end{proof}


\section{Dirichlet forms}
\label{Chap-Dir}

In the following, we provide a short summary on Dirichlet forms. The more interested reader is referred to \cite{Fukushima80,BouleauHirsch91,MaRoeckner92,FukushimaOshimaTakeda94,Davies95} for further background on Dirichlet forms.

Let $X$ be a locally compact, connected, separable metric space with positive measure $m$. Consider a dense subspace $\Dd\subseteq L^2(X,m)$ and a sesquilinear, nonnegative, symmetric map $\Ee:\Dd\times\Dd\to\RM$. If $\Dd$ is closed in the energy norm $\|\cdot\|_\Ee^2:=\Ee(\cdot,\cdot)+\|\cdot\|^2$, then $\Ee$ is called a {\em closed (symmetric) form} in $L^2(X,m)$, where $\|\cdot\|$ denotes the $L^2$-norm. We associate to $\Ee$ the unique self-adjoint operator $(H,D(H))$  satisfying $D(H)\subseteq\Dd$ and $\Ee(\psi,u)=\langle H\psi,u\rangle$ for all $\psi\in D(H)$ and $u\in\Dd$. We use the short notation $\Ee(u)=\Ee(u,u)$ for $u\in\Dd$.

\medskip

The (symmetric) closed form $\Ee$ is called {\em (symmetric) Dirichlet form} whenever $T\circ u\in\Dd$ and $\Ee(T\circ u)\leq \Ee(u)$ holds for all $u\in\Dd$ and every normal contraction $T:\RM\to\RM$. Recall that $T:\RM\to\RM$ is called {\em normal contraction} if $T(0)=0$ and $|T(x)-T(y)|\leq |x-y|$ for each $x,y\in\RM$. A Dirichlet form is called {\em regular} if $\Dd\cap C_0(X)$ is dense in $(\Dd,\|\cdot\|_{\mathcal{E}})$ and in $(C_0(X),\|\cdot\|_\infty)$. Furthermore, a Dirichlet form is called {\em strongly local} if $\Ee(u,v)=0$ whenever $u\in\Dd$ is constant a.s. on the support of $v\in\Dd$.

\medskip

The Beurling-Deny formula asserts that for every symmetric regular Dirichlet form $\Ee$ there exists a unique decomposition
$$
\Ee(u,v) \;
	= \;\Ee^{(a)}(u,v)+\Ee^{(b)}(u,v)+\Ee^{(c)}(u,v)
	\qquad
	u,v\in\Dd\,,
$$
where $\Ee^{(a)}$ is the killing term, $\Ee^{(b)}$ is the jump part and $\Ee^{(c)}$ is the strongly local part. Each of this parts can be represented as
$$
\Ee^{(\ast)}(u,v) \;
	= \; \int_X  \dmu^{(\ast)}(u,v)
	\qquad
	\ast\in\{a,b,c\},
$$
for a suitable Radon measure $\mu^{(\ast)}$. For  $\ast\in\{a,b,c\}$, the  measure  $\mu^{(\ast)}$ satisfies the Cauchy-Schwarz inequality
$$
\left|\int_X \psi \varphi  \dmu^{(\ast)}(u,v)\right| \;
	\leq \; \left(\int_X |\psi|^2 \dmu^{(\ast)}(u,u) \right)^{\frac{1}{2}}
		\left(\int_X |\varphi|^2 \dmu^{(\ast)}(v,v) \right)^{\frac{1}{2}}.
$$
Given a regular Dirichlet form, it extends to its local form domain, see \cite[Section~3.1]{FrLeWi14} for details. An element $u$ in its local domain is called a {\em generalized eigenfunction with eigenvalue $\lambda\in\RM$} if $\Ee(u,v)=\lambda\langle u,v\rangle$ for all $v\in\Dd\cap C_0(X)$.

\medskip

In order to generalize Proposition~\ref{Prop-Caccio} and Proposition~\ref{Prop-Shnol} to regular Dirichlet forms, one needs additionally an integrated version of the Leibniz rule \cite[Theorem~3.7]{FrLeWi14} and a suitable notion of a ground state and null-sequence. The ground state transform is established for a class of strictly local regular Dirichlet forms in \cite{LeStVe09,LeStVe11}. Additionally, the authors show the existence of a weak positive solution (Allegretto-Piepenbrink-type theorem) by approximating such a solution under suitable assumptions on the Dirichlet form. In the following, we focus on regular Dirichlet forms over a discrete set $X$ since the existence of a ground state with corresponding null-sequence is known under reasonable assumption, see \cite{KeLe12,KePiPo16}. The statements of Theorem~\ref{t:char_criticality} stay valid in this setting while some estimates need to be adjusted since the chain rule does not hold for discrete Laplacians, see \cite{KePiPo16}.

\medskip

Let $X$ be a countable infinite set equipped with the discrete topology, and $m:X\to(0,\infty)$ be a measure on $X$ with full support. Then $\ell^2(X,m)$ is defined by all $u:X\to\RM$ satisfying $\sum_{x\in X}u(x)^2\, m(x)<\infty$ equipped with the standard scalar product weighted by $m$. A symmetric weighted graph over $X$ is a map $b:X\times X\to[0,\infty)$ satisfying $b(x,x)=0\,,\; x\in X\,,$ $b(x,y)=b(y,x)\,,\; x,y\in X$ and $\sum_{y\in X} b(x,y) <\infty$ for all $x\in X$. We say that $x,y\in X$ are \emph{adjacent} or \emph{neighbors} or \emph{connected by an edge}
if $b(x,y)>0$. A graph $b$ is called {\em connected} if there is a path connecting every two vertices $x,y\in X$. According to \cite[Theorem~7]{KeLe12}, there exists for every regular Dirichlet form $Q$ on $\ell^2(X,m)$ a graph $b$ on $X$ with nonnegative $c:X\to[0,\infty]$ such that $Q=Q_{b,c}$, where
$$
Q_{b,c}(u,v) \;
	:= \; \sum_{x,y\in X} b(x,y)\, (u(x)-u(y))\, (v(x)-v(y)) + \sum_{x\in X} u(x)\, v(x)\, c(x)
$$
defined on $\Dd(Q_{b,c}):=\overline{C_0(X)}^{\|\cdot\|_{Q_{b,c}}}$. Note that $C_0(X)$ is the set of all real-valued functions of $X$ with finite support. Denote by $H$ the unique operator associated with $Q$. The integrated version of the Leibniz rule reads as follows
$$
\sum_{x,y\in X} d_b(uv)(x,y)\, d_b w(x,y)
	= \sum_{x,y\in X}  u(x)\, d_b v(x,y)\, d_b(w)(x,y) + \sum_{x,y\in X} v(x)\, d_b u(x,y)\, d_b(w)(x,y)
$$
where $d_b u(x,y):= \sqrt{b(x,y)}\, (u(x)-u(y))$, see \cite[Theorem~3.9]{FrLeWi14} or \cite[Lemma~3.2]{HuaMasKelWoj13}.

\medskip

Let $F(X)$ be the set of all $u:X\to\RM$ satisfying $\sum_{y\in X} b(x,y) |f(y)|<\infty$ for each $x\in X$. A function $u\in F(X)$ is called {\em harmonic} if $Hu=0$ in $X$. Let $Q\geq 0$ on $C_0(X)$. Every function $W:X\to\RM$ defines a quadratic form $Q_W$ on $C_0(X)$ by $Q_W(u):=\sum_{x\in X} W(x) u(x)^2$. Analogously to the continuous setting, $Q\geq 0$ on $C_0(X)$ is said to be {\em subcritical} in $X$ if there is a nonzero $W:X\to [0,\infty)$ such that $Q-Q_W\geq 0$ on $C_0(X)$. Furthermore, a sequence $(\varphi_n)$ of nonnegative functions in $C_0(X)$ is called {\em null-sequence} if there is an $o\in X$ and a constant $C>0$ such that $\varphi_n(o)=C$ for each $n\in\NM$ and $Q(\varphi_n)\to 0$. Note that for a nonnegative quadratic form $Q$ the existence of a positive harmonic function is guaranteed in the case of locally finite graphs \cite{HaesKe11}, or in the critical case \cite{KePiPo16}. In the following, the space of bounded real-valued functions on $X$ equipped with the uniform norm is denoted by $B(X,\RM)$. With this at hand, we can show the discrete analog of Proposition~\ref{Prop-Caccio} and Proposition~\ref{Prop-Shnol}.

\medskip

Let $Q:=Q_{b,0}$ be a regular Dirichlet form on $\ell^2(X,m)$ with zero potential $c\equiv 0$ and $L$ the associated self-adjoint operator. For $W\in B(X,\RM)$, the form associated with $L+W$ is denotes by $Q_{b,W}=Q+Q_W$ which is semibounded with constant $c:=-\|W\|_\infty$ as $Q$ is nonnegative. Let $\|v\|_{Q_{b,W}} := Q(v,v) + (1+\|W\|_\infty) \|v\|^2$.
We have the following Caccioppoli-type estimate.

\begin{proposition}
\label{Prop-DiscrCaccio}
Suppose that $Q:=Q_{b,0}$ is a regular Dirichlet form on $\ell^2(X,m)$ with zero potential $c\equiv 0$ and $L$ the associated self-adjoint operator. For $W\in B(X,\RM)$, consider a generalized eigenfunction $u\in F(X)$ of $L+W$ with eigenvalue $\lambda\in\RM$, such that $|u|\leq 1$ in $X$. Then
$$
\sum_{x,y\in X} v(x)^2 d_bu(x,y)^2 \;
	\leq \; \Big(1 + \sqrt{1+|\lambda|+\|W\|_\infty}\Big)^2
$$
holds for every $v\in\Dd(Q_{b,W})$ satisfying $\|v\|_{Q_{b,W}}\leq 1$.
\end{proposition}
\begin{proof}
Let $v\in C_0(X)$ be such that $\|v\|_{Q_{b,W}}\leq 1$. A short computation and the symmetry $b(x,y)=b(y,x)$ leads to
\begin{equation}\label{eq11}
    \sum_{x,y\in X} |d_b v^2(x,y)|  |d_b u(x,y)|
	\leq \; 2 \sum_{x,y\in X} |v(x)| |d_b v(x,y)| |d_b u(x,y)|,
\end{equation}
where each of the sums is finite as $v$ has finite support. Set
$$z:=\sqrt{\sum_{x,y\in X} v(x)^2\, d_bu(x,y)^2}.$$
Like in Proposition~\ref{Prop-Caccio}, $\|v\|_{Q_{b,W}}\leq 1$ leads to $\|v\|^2\leq 1$ and $Q(v,v)\leq 1$ since $\langle W \, v,v \rangle + \|W\|_\infty \|v\|^2\geq 0$. Using that $u$ is a generalized eigenfunction and the Leibniz rule, the estimate
$$
z^2
	= \left|
		Q\big(u,v^2u\big)
		-
		\sum_{x,y\in X} u(x)\, d_b v^2(x,y)\, d_b u(x,y)
	\right|
	\leq |\lambda| + \|W\|_\infty + \sum_{x,y\in X} |d_b v^2(x,y)|\, |d_b u(x,y)|
$$
follows by using $|u|\leq 1$ and $\|v\|_{Q_{b,W}}\leq 1$. Hence, $z^2\leq |\lambda| + \|W\|_\infty + 2\, z$ is concluded using \eqref{eq11} and the Cauchy-Schwarz inequality. Following the lines of the proof of Proposition~\ref{Prop-Caccio} the desired result is derived.
\end{proof}
\begin{proposition}
\label{Prop-DiscrShnol}
Suppose $Q:=Q_{b,0}$ is a regular Dirichlet form on $\ell^2(X,m)$ with zero potential $c\equiv 0$ such that the associated self-adjoint operator $L$ is critical with ground state $1$. Let $W\in B(X,\RM)$ and $u\in F(X)$ be a generalized eigenfunction of $L+W$ with eigenvalue $\lambda\in\RM$. If $|u|\leq 1$, then $\lambda\in\sigma(L+W)$.
\end{proposition}
\begin{proof}
According to \cite[Theorem~2.20]{KePiPo16}, there exists a null-sequence $(\varphi_n)$ of $Q$ satisfying $0\leq \varphi_n\leq 1$ such that $\varphi_n(x)\to 1$ holds for each $x\in X$. We may suppose that $u\not\in L^2(X,m)$ since otherwise $\lambda\in\sigma(L+W)$ follows immediately. Define the sequence $w_n:=\frac{\varphi_n u}{\|\varphi_n u\|}\,,\; n\in\NM$.
The form associated with $L+W$ is denoted by $Q_{b,W}=Q + Q_W$. Using the integrated version of the Leibniz rule and that $u$ is a generalized eigenfunction of $L+W$, we have
\begin{multline*}
|Q_{b,W}(w_n,v)-\lambda \langle w_n, v\rangle|\\
=\frac{1}{\|\varphi_n u\|}
\left|
	\sum_{x,y\in X} u(x) \, 1 \, d_b\varphi_n(x,y) \, d_b v(x,y)
	-
	\sum_{x,y\in X} v(x) \, 1 \, d_b u(x,y) \, d_b\varphi_n(x,y)
\right|
\end{multline*}
Following the lines of the proof of Proposition~\ref{Prop-Shnol} and using Cauchy-Schwarz, we get
$$
\Big| Q_{b,W}(w_n,v)-\lambda  \langle w_n,v\rangle \Big| \;
	\leq \; \frac{Q(\varphi_n,\varphi_n)^{\frac{1}{2}}}{\|\varphi_n u\|}
			\left(
				1 + \left(\sum_{x,y\in X} v(x)^2 \, d_bu(x,y)^2\right)^{\frac{1}{2}}
			\right)
$$
for each $v\in\Dd(Q_{b,W})$ satisfying $\|v\|_{Q_{b,W}}\leq 1$. Due to Proposition~\ref{Prop-DiscrCaccio}, $\sum_{x,y\in X} v(x)^2 \, d_bu(x,y)^2$ is bounded by a constant (only depending on $\|W\|_\infty$ and $\lambda$). Furthermore, there exists an $n_0\in\NM$ such that $\|\varphi_n u\|\geq 1$ for $n\geq n_0$. With this at hand, the previous considerations lead to
$$
\lim_{n\to\infty} \;
	\sup_{v\in\Dd(Q_{b,W}), \|v\|_{Q_{b,W}}\leq 1}
		\Big|
			Q_{b,W}(w_n,v)-\lambda  \langle w_n,v\rangle
		\Big|
			= 0
$$
as $(\varphi_n)$ is a null-sequence of $Q$. Thus, Proposition~\ref{Prop-CritWeyl} yields $\lambda\in\sigma(L+W)$.
\end{proof}

\begin{theorem}
\label{Theo-DiscrShnol}
Let $b$ be a connected graph, $c:X\to\RM$ and $Q=Q_{b,c}$ be a regular Dirichlet form on $X$ with associated self-adjoint operator $H$. Suppose that there is a $W\in C_b(X)$ such that $H+W$ is critical in $X$ with ground state $\varphi$. If $u\in F(X)$ is a generalized eigenfunction of $H$ with eigenvalue $\lambda\in\RM$ and $|u|\leq \varphi$ in $X$, then $\lambda\in\sigma(H)$.
\end{theorem}

\begin{proof}
The ground state transform $T_\varphi$ eliminates the zero-order term $c$ and $T_\varphi\circ (H+W)\circ T_\varphi^{-1}$ is critical with ground state $1$, see \cite[Proposition~2.7]{KePiPo16}. With this at hand the proof follows the same lines as the proof of Theorem~\ref{Theo-Shnol} by using Proposition~\ref{Prop-DiscrShnol} and $\sigma(H)=\sigma(H_\varphi)$.
\end{proof}
 \begin{center}{\bf Acknowledgments} \end{center}
The authors wish to thank Daniel Lenz and Marcel Schmidt for pointing out the papers \cite{BoLeSt09,LeStVe09,LeStVe11}. They acknowledge the support of the Israel Science Foundation (grants No. 970/15) founded by the Israel Academy of Sciences and Humanities.


\providecommand{\bysame}{\leavevmode\hbox to3em{\hrulefill}\thinspace}
\providecommand{\MR}{\relax\ifhmode\unskip\space\fi MR }
\providecommand{\MRhref}[2]{%
  \href{http://www.ams.org/mathscinet-getitem?mr=#1}{#2}
}
\providecommand{\href}[2]{#2}

\end{document}